\documentclass[11pt]{amsart}

\usepackage{amsmath}
\usepackage{amsthm}
\usepackage{amssymb,graphics,amscd,amsfonts}
\usepackage{color}
\usepackage{comment}
\usepackage{mathrsfs}
\usepackage{mathtools}
\usepackage{enumerate}

\usepackage{float}
\usepackage {booktabs}
\usepackage{color,float}
\usepackage{lscape}
\theoremstyle{plain}
\newtheorem{theorem}{Theorem}[section]

\newtheorem{lemma}[theorem]{Lemma}

\newtheorem{example}[theorem]{Example}
\newtheorem{definition}[theorem]{Definition}
\newtheorem{remark}[theorem]{Remark}

\newtheorem*{problem*}{Problem}

\newtheorem{THEOREM}{Theorem}

\newtheorem{COROLLARY}[THEOREM]{Corollary}

\theoremstyle{definition}
\newtheorem*{acknowledgments}{Acknowledgments}

\numberwithin{equation}{section}
\allowdisplaybreaks

\DeclareMathOperator{\SL}{SL}

\DeclareMathOperator{\DF}{DF}
\DeclareMathOperator{\diag}{diag}
\DeclareMathOperator{\Pic}{Pic}

\makeatletter
\@namedef{subjclassname@2020}{\textup{2020} Mathematics Subject Classification}
\makeatother

\title{
K-instability of hyperplane sections of Segre Varieties
}
\author{Shunsuke Saito}
\address{Department of Mathematics, Faculty of Science, Tokyo University of Science, 1-3 Kagurazaka, Shinjuku-ku, Tokyo 162-8601, Japan}
\email{saito@rs.tus.ac.jp}
\begin{document}
\begin{abstract}
We prove that a normal hyperplane section of the Segre variety $\Sigma_{m, n}$ is K-unstable with respect to any polarization if $m\neq n$ or it is not smooth. 
\end{abstract}
\maketitle
\setcounter{tocdepth}{1}
We work over an algebraically closed field of characteristic zero unless otherwise stated. 
\section{Introduction}
The existence of a constant scalar curvature K\"ahler metric (cscK metric for short) in a given K\"ahler class has been a central problem in K\"ahler geometry. Unfortunately, not all compact K\"ahler manifolds admit cscK metrics. 
A theorem of Matsushima \cite{Matsushima} and Lichnerowicz \cite{L} states that the automorphism group of a compact K\"ahler manifold admitting a cscK metric must be reductive. 
Hano \cite{Hano} proved that a smooth hyperplane section of the Segre variety $\Sigma_{m, n}$, the image of the Segre embedding $\mathbf{P}^{m}\times\mathbf{P}^{n}\hookrightarrow \mathbf{P}^{mn+m+n}$, has non-reductive automorphism group if $m\neq n$. Thus such a hyperplane section does not admit cscK metrics in any K\"ahler class. 

The celebrated Yau-Tian-Donaldson conjecture asserts that a polarized manifold admits a cscK metric if and only if it is K-polystable. 
The ``only if'' part was proved by Stoppa \cite{Stoppa}, Mabuchi \cite{Mabuchi}, Berman-Darvas-Lu \cite{BDL}, Dyrefelt \cite{Dyrefelt}, Hisamoto \cite{Hisamoto}, and others for various stabilities. The reverse direction is considered to be wide open in general. 
In light of the conjecture, a smooth hyperplane section of $\Sigma_{m, n}$ should be K-unstable when $m\neq n$. 
However we can not conclude this assertion due to the lack of an algebro-geometric counterpart to the Lichnerowicz theorem, specifically the reductivity of the automorphism group of a K-polystable polarized manifold (except for the anticanonically polarized Fano case \cite{ABHLX}, which corresponds to the Matsushima theorem). 

In this paper, we directly prove the K-instability of certain hyperplane sections of $\Sigma_{m, n}$, which is our main theorem. 

\begin{THEOREM}\label{K}
Let $X$ be a normal hyperplane section of the Segre variety $\Sigma_{m, n}$. 
If $m\neq n$ or $X$ is not smooth, then $X$ is K-unstable with respect to any polarization. 
\end{THEOREM}

This is an algebro-geometric analogue of Hano's result. 
Our formulation has an advantage that hyperplane sections are allowed to be singular. 

To show Theorem \ref{K}, we need a formula for the Donaldson-Futaki invariant of a product test configuration. 
For notational simplicity, hyperplane sections are treated as hypersurfaces in the product of projective spaces. 
\begin{THEOREM}\label{DF}
Let $X$ be a normal hypersurface of bidegree $(1, 1)$ in $\mathbf{P}^{m}\times\mathbf{P}^{n}$ and $f$ be its defining polynomial. Let $\lambda$ be a one-parameter subgroup of $\SL(m+1)\times\SL(n+1)$ which preserves $X$ and $\alpha$ be the weight of $f$ for $\lambda$. For positive integers $d$ and $e$, the Donaldson-Futaki invariant of the product test configuration $(\mathcal{X}_{\lambda}, \mathcal{L}_{\lambda})$ for $(X, \mathcal{O}_{X}(d, e))$ induced by $\lambda$ is given by
\[
\DF(\mathcal{X}_{\lambda}, \mathcal{L}_{\lambda})=-\dfrac{2mnde}{(me+nd)^{2}}\alpha. 
\]
\end{THEOREM}

Theorem \ref{K} will be proved by combining of Theorem \ref{DF} with the construction of a one-parameter subgroup $\lambda$ satisfying $\alpha=1$. 
When the field we work over is the complex number field, this argument also shows the following result on the Futaki invariant \cite{F1, F2} (see also \cite{B} and \cite{C}). 
\begin{COROLLARY}
Let $X$ be a smooth hyperplane section of $\Sigma_{m, n}$. If $m+n\ge 4$ and $m\neq n$, then the Futaki invariant of $X$ does not vanish for any rational K\"ahler class. 
\end{COROLLARY}
By \cite[Theorem 4]{C}, this implies that such a hyperplane section has a chance to admit a non-trivial extremal K\"ahler metric in a rational K\"ahler class. Note that  in the case when $m+n=3$, $X$ is isomorphic to the first Hirzebruch surface as seen in the proof of Theorem \ref{K} and admits non-trivial extremal K\"ahler metrics in every K\"ahler class due to Guan \cite{G} and Hwang \cite{Hwang}. 

\section{Donaldson-Futaki invariants and K-instability}
This section reviews the definitions of Donaldson-Futaki invariants and K-instability. 
For further details, we refer to \cite{BHJ}. 

Throughout this section, $(X, L)$ denotes an $n$-dimensonal polarized normal variety. 
\begin{definition}
A test configuration for $(X, L)$ is a quadruple of the following data: 
\begin{itemize}
\item a normal $\mathbf{G}_{m}$-variety $\mathcal{X}$; 
\item a $\mathbf{G}_{m}$-equivariant flat proper morphism $\pi\colon \mathcal{X}\to\mathbf{A}^{1}$, where $\mathbf{G}_{m}$ acts on $\mathbf{A}^{1}$ by the standard multiplication; 
\item a $\mathbf{G}_{m}$-linearized ample $\mathbf{Q}$-invertible sheaf $\mathcal{L}$ on $\mathcal{X}$; 
\item an isomorphism $(\mathcal{X}_{1}, \mathcal{L}_{1})\simeq (X, L)$. 
\end{itemize}
We say that $(\mathcal{X}, \mathcal{L})$ is product if $\mathcal{X}\simeq X\times\mathbf{A}^{1}$. 
\end{definition}

\begin{example}
Fix a $\mathbf{G}_{m}$-action $\lambda$ on $X$. This induces the diagonal $\mathbf{G}_{m}$-action on $X\times\mathbf{A}^{1}$ as the product of $\lambda$ and the standard action on $\mathbf{A}^{1}$. 
Since some power of $L$ admits a $\mathbf{G}_{m}$-linearization, we obtain a product test configuration for $(X, L)$. 
We denote by $(\mathcal{X}_{\lambda}, \mathcal{L}_{\lambda})$ the test configuration induced by $\lambda$. 
\end{example}

Consider a test configuration $(\mathcal{X}, \mathcal{L})$ for $(X, L)$. 
Let $w_{k}$ denote the total weight of the dual action of $\mathbf{G}_{m}$ on $H^{0}(\mathcal{X}_{0}, \mathcal{L}_{0}^{\otimes k})$, and let $N_{k}=\dim H^{0}(\mathcal{X}_{0}, \mathcal{L}_{0}^{\otimes k})$. 
For a sufficiently divisible positive integer $k$, we have the following asymptotic expansions: 
\begin{align*}
N_{k}&=a_{0}k^{n}+a_{1}k^{n-1}+\cdots, \\
w_{k}&=b_{0}k^{n+1}+b_{1}k^{n}+\cdots. 
\end{align*}
\begin{definition}
The Donaldson-Futaki invariant of $(\mathcal{X}, \mathcal{L})$ is defined by
\[
\DF(\mathcal{X}, \mathcal{L})=2\dfrac{
a_{1}b_{0}-a_{0}b_{1}
}{a_{0}^{2}}. 
\]
\end{definition}
Note that the Donaldson-Futaki invariant remains unchanged when the linearization on $\mathcal{L}$ is twisted. 

\begin{remark}
Assume that the field we work over is the complex number field. 
When $X$ is smooth, the Donaldson-Futaki invariant of the product test configuration $(\mathcal{X}_{\lambda}, \mathcal{L}_{\lambda})$ induced by a $\mathbf{G}_{m}$-action $\lambda$ on $X$ is a non-zero constant multiple of the Futaki invariant of $X$ for $c_{1}(L)$ evaluated at the generator of $\lambda$. We refer to \cite[Proposition 2.2.2]{Donaldson} for the proof. 
\end{remark}

\begin{definition}
A polarized normal variety $(X, L)$ is called
\begin{itemize}
\item K-semistable if $\DF(\mathcal{X}, \mathcal{L})\ge 0$ for all test configurations $(\mathcal{X}, \mathcal{L})$ for $(X, L)$. 
\item K-unstable if $(X, L)$ is not K-semistable. 
\end{itemize}
\end{definition}
\section{Proof of the main results}
\subsection{A formula for Donaldson-Futaki invariants}
\begin{proof}[Proof of Theorem \ref{DF}]
Set
\[
S_{a, b}= H^{0}(\mathbf{P}^{m}\times\mathbf{P}^{n}, \mathcal{O}(a, b)), \quad R_{c}= H^{0}(X, \mathcal{O}_{X}(d, e)^{\otimes c})
\]
for non-negative integers $a$, $b$, and $c$. In this proof, $k$ always denotes a sufficiently divisible positive integer. 

We have an exact sequence
\begin{align*}
0\longrightarrow S_{dk-1, ek-1}\overset{\cdot f}{\longrightarrow}S_{dk, ek}\longrightarrow R_{k}\longrightarrow 0, 
\end{align*}
where the second morphism is multiplication by $f$ and the third is the restriction onto $X$. 
Using this, we obtain
\begin{align*}
N_{k}
&=\dim R_{k}=\dim S_{dk, ek}-\dim S_{dk-1, ek-1}\\
&=\binom{m+dk}{m}\binom{n+ek}{n}-\binom{m+dk-1}{m}\binom{n+ek-1}{n}\\
&=\dfrac{d^{m-1}e^{n-1}}{m!n!}(me+nd)k^{m+n-1}\\
&\qquad+\dfrac{d^{m-2}e^{n-2}}{2m!n!}(
m^{2}(m-1)e^{2}+mn(m+n)de+n^{2}(n-1)d^{2})k^{m+n-2}+\cdots
\end{align*}
and
\[
a_{0}=\dfrac{d^{m-1}e^{n-1}}{m!n!}(me+nd), \quad
a_{1}=\dfrac{d^{m-2}e^{n-2}}{2m!n!}(
m^{2}(m-1)e^{2}+mn(m+n)de+n^{2}(n-1)d^{2}). 
\]
Next we compute the total weight $w_{k}$ of the $\mathbf{G}_{m}$-action on $R_{k}$ induced by $\lambda$. 
Since $\lambda$ is a one-parameter subgroup of $\SL(m+1)\times\SL(n+1)$, the total weights of the actions on $S_{dk-1, ek-1}$ and $S_{dk, ek}$ are both zero. 
Combining this with the exact sequence above, we obtain
\begin{align*}
w_{k}&=-\alpha \dim S_{dk-1, ek-1}
=-\alpha\binom{m+dk-1}{m}\binom{n+ek-1}{n}\\
&=-\dfrac{d^{m}e^{n}}{m!n!}\alpha k^{m+n}
-\dfrac{d^{m-1}e^{n-1}}{2m!n!}(m(m-1)e+n(n-1)d)\alpha k^{m+n-1}+\cdots
\end{align*}
and 
\[
b_{0}=-\dfrac{d^{m}e^{n}}{m!n!}\alpha, \quad
b_{1}=-\dfrac{d^{m-1}e^{n-1}}{2m!n!}(m(m-1)e+n(n-1)d)\alpha. 
\]
The formula for $\DF(\mathcal{X}_{\lambda}, \mathcal{L}_{\lambda})$ now follows from a straightforward computation. 
\end{proof}
\subsection{K-instability of hyperplane sections}
Let $X$ be a hyperplane section of the Segre variety $\Sigma_{m, n}$. In what follows, we treat $X$ as a hypersurface of bidegree $(1, 1)$ in $\mathbf{P}^{m}\times\mathbf{P}^{n}$. We choose homogeneous coordinates on $\mathbf{P}^{m}$ and $\mathbf{P}^{n}$ so that the defining polynomial $f$ of $X$ is written as
\[
f(x_{0}, \ldots, x_{m}, y_{0}, \ldots, y_{n})=\sum_{i=0}^{r}x_{i}y_{i}, 
\]
where $0\le r\le \min\{m, n\}$. 
The Jacobian criterion for smoothness then provides the following. 
\begin{lemma}\label{lem}
\begin{enumerate}
\item $X$ is smooth if and only if $r=\min\{m, n\}$. 
\item $X$ is normal if and only if $r\ge 1$. 
\end{enumerate}
\end{lemma}
\begin{proof}[Proof of Theorem \ref{K}]
Assume that $m\le n$. Suppose that $m\neq n$ or $X$ is not smooth. Then $m+n\ge 3$ by Lemma \ref{lem}. 

If $m+n=3$, then $r=1$ again by Lemma \ref{lem}. $X$ is then isomorphic to the first Hirzebruch surface, which was proved to be slope unstable and hence K-unstable with respect to any polarization by Ross-Thomas \cite[Example 5.27]{RT06}. 

Assume that $m+n\ge 4$. Pick an ample invertible sheaf $L$ on $X$. 
By the Grothendieck-Lefschetz hyperplane theorem \cite[Expos\'e XII, Corollaire 3.6]{SGA}, the restriction $\Pic(\mathbf{P}^{m}\times\mathbf{P}^{n})\to \Pic(X)$ is an isomorphism in this case. We thus have $L\simeq \mathcal{O}_{X}(d, e)$ for some positive integers $d$ and $e$. 
Under our assumption, we have $r<n$ by Lemma \ref{lem} and can define a one-parameter subgroup $\lambda$ of $\SL(m+1)\times\SL(n+1)$ by
\[
\lambda(t)=(
\diag(1, \dots, 1), \diag(\underbrace{t^{-1}, \dots, t^{-1}}_{r+1}, t^{r+1}, 1, \dots, 1)). 
\]
One can easily see that the weight of $f$ with respect to $\lambda$ is 1. 
We thus obtain 
\[
\DF(\mathcal{X}_{\lambda}, \mathcal{L}_{\lambda})=-\frac{2mnde}{(me+nd)^{2}}<0
\]
 by Theorem \ref{DF} and the K-instability of $(X, L)$. 
\end{proof}

\begin{acknowledgments}
This work was partially supported by JSPS KAKENHI Grant Number JP20K14321. 
\end{acknowledgments}

\bibliographystyle{amsalpha}

\begin{thebibliography}{A}
\bibitem{ABHLX}
J.~Alper, H.~Blum, D.~Halpern-Leistner, and C.~Xu, 
Reductivity of the automorphism group of K-polystable Fano varieties. 
\emph{Invent. Math.} \textbf{222} (2020), 995--1032. 
\bibitem{B} S.~Bando, 
An obstruction for Chern class forms to be harmonic. 
\emph{Kodai Math. J.} \textbf{29} (2006), 337--345. 
\bibitem{BDL} R.~J.~Berman, T.~Darvas, and C.~H.~Lu, 
Regularity of weak minimizers of the K-energy and applications to properness and K-stability. 
\emph{Ann. Sci. \'Ec. Norm. Sup\'er.} \textbf{53} (2020), 267--289. 
\bibitem{BHJ} S.~Boucksom, T.~Hisamoto, and M.~Jonsson, 
Uniform K-stability, Duistermaat-Heckman measures and singularities of pairs. 
\emph{Ann. Inst. Fourier (Grenoble)} \textbf{67} (2017), 743--841.
\bibitem{C} E.~Calabi, 
Extremal K\"ahler metrics II. 
\emph{Differential geometry and complex analysis}, Springer-Verlag, Berlin, 1985, pp.95--114. 
\bibitem{Donaldson} S.~K.~Donaldson, 
Scalar curvature and stability of toric varieties.
\emph{J. Differential Geom.} \textbf{62} (2002), 289--349.
\bibitem{Dyrefelt} Z.~S.~Dyrefelt, 
On K-polystability of cscK manifolds with transcendental cohomology class (with an appendix by R.~Dervan). 
\emph{Int. Math. Res. Not. IMRN} \textbf{9} (2020), 2769--2817.
\bibitem{F1} A.~Futaki, 
An obstruction to the existence of Einstein K\"ahler metrics. 
\emph{Invent. Math.} \textbf{73} (1983), 437--443. 
\bibitem{F2} A.~Futaki, 
On compact K\"ahler manifolds of constant scalar curvatures. 
\emph{Proc. Japan Acad. Ser. A Math. Sci.} \textbf{59} (1983), 401--402. 
\bibitem{SGA} A.~Grothendieck, 
\emph{Cohomologie locale des faisceaux coh\'erents et th\'eor\`emes de Lefschetz locaux et globaux ({SGA} 2)}. 
Documents Math\'ematiques (Paris) \textbf{4} Soci\'et\'e{} Math\'ematique de France, Paris, 2005. 
\bibitem{G} D.~Guan, Existence of extremal metrics on compact almost homogeneous K\"ahler manifolds with two ends. 
\emph{Trans. Amer. Math. Soc.} \textbf{347} (1995), 2255--2262. 
\bibitem{Hano} J.~Hano, Examples of projective manifolds not admitting K\"ahler metric with constant scalar curvature. \emph{Osaka J.~ Math.} \textbf{20} (1983), 787--791. 
\bibitem{Hisamoto} T.~Hisamoto, 
Stability and coercivity for toric polarizations. 
arXiv:1610.07998. 
\bibitem{Hwang} A.~D.~Hwang, 
On existence of K\"ahler metrics with constant scalar curvature. 
\emph{Osaka J. Math.} \textbf{31} (1994), 561--595. 
\bibitem{L} A.~Lichnerowicz, 
Sur les transformations analytiques des vari\'et\'es k\"ahl\'eriennes compactes. 
\emph{C. R. Acad. Sci. Paris} \textbf{244} (1957), 3011--3013. 
\bibitem{Mabuchi} T.~Mabuchi, 
A stronger concept of K-stability. 
arXiv:0910.4617. 
\bibitem{Matsushima} Y.~Matsushima, 
Sur la structure du groupe d'hom\'eomorphismes analytiques d'une certaine vari\'et\'e k\"ahl\'erienne. 
\emph{Nagoya Math. J.} \textbf{11} (1957), 145--150.
\bibitem{RT06} J.~Ross and R.~Thomas, 
An obstruction to the existence of constant scalar curvature K\"ahler metrics. 
\emph{J. Differential Geom.} \textbf{88} (2011), 109--159. 
\bibitem{Stoppa} J.~Stoppa, 
K-stability of constant scalar curvature K\"ahler manifolds. 
\emph{Adv. Math.} \textbf{221} (2009), 1397--1408.
\end{thebibliography}

\end{document}